\documentclass{amsart}

\newtheorem{theorem}{Theorem}[section]
\newtheorem{lemma}[theorem]{Lemma}
\newtheorem{corollary}[theorem]{Corollary}
\newtheorem{OldTheorem}{Theorem}

\newtheorem{prop}{\rm \quad}

\numberwithin{equation}{section}

\def\BV{{\rm BV\,}}
\def\OSC{{\rm OSC\,}}
\def\ZZ{\ensuremath{\mathbb Z}}
\def\ZI{\ensuremath{\mathbb I}}
\def\ZN{\ensuremath{\mathbb N}}

\def\zP{\ensuremath{\mathcal P}}
\def\ZT{\ensuremath{\mathbb T}}
\def\ZR{\ensuremath{\mathbb R}}
\def\md#1#2\emd{\ifx0#1
\begin{equation*} #2 \end{equation*}\fi  
\ifx1#1\begin{equation}#2\end{equation}\fi   
\ifx2#1\begin{align*}#2\end{align*}\fi   
\ifx3#1\begin{align}#2\end{align}\fi    
\ifx4#1\begin{gather*}#2\end{gather*}\fi  
\ifx5#1\begin{gather}#2\end{gather}\fi   
\ifx6#1\begin{multline*}#2\end{multline*}\fi  
\ifx7#1\begin{multline}#2\end{multline}\fi  
\ifx8#1\begin{multline*}\begin{split}#2\end{split}\end{multline*}\fi
\ifx9#1\begin{multline}\begin{split}#2\end{split}\end{multline}\fi
}
\newcommand {\e }[1]{(\ref{#1})}
\newcommand {\lem }[1]{Lemma \ref{#1}}
\newcommand {\trm }[1]{Theorem \ref{#1}}

\begin{document}

\title[On generalizations of Fatou's theorem]{On generalizations of Fatou's theorem for the integrals with general kernels}%
\author{G. A. Karagulyan }

\address{Institute of Mathematics of Armenian
National Academy of Sciences, Baghramian Ave.- 24/5, 0019,
Yerevan, Armenia}%
\email{g.karagulyan@yahoo.com}%
\author{ M. H. Safaryan}
\address{Yerevan State University,
Alek Manukyan, 1, 0049,
Yerevan, Armenia}%
\email{mher.safaryan@gmail.com }

\begin{abstract}
We define $\lambda(r)$-convergence, which is a generalization of nontangential convergence in the unit disc. We prove Fatou-type theorems on almost everywhere nontangential convergence of Poisson-Stiltjes integrals for general kernels $\{\varphi_r\}$, forming an approximation of identity. We prove that the bound
\md0
\limsup_{r\to 1}\lambda(r) \|\varphi_r\|_\infty<\infty
\emd
is necessary and sufficient for almost everywhere $\lambda(r)$-convergence of the integrals
\md0
\int_\ZT \varphi_r(t-x)d\mu(t).
\emd
\end{abstract}
\maketitle
\section{Introduction}
In his famous paper Fatou \cite{Fat} proved
\begin{OldTheorem}[Fatou, 1906]
Any bounded analytic function on the unit disc $D=\{z\in \mathbb{C}:\, |z|<1\}$ has nontangential limit for almost all boundary points.
\end{OldTheorem}
\begin{OldTheorem}[Fatou, 1906]
If a function of bounded variation $\mu(t)$ is differentiable at $x_0\in \ZT$, then the Poisson integral
\md0
\zP_r(x,d\mu)=\frac{1}{2\pi }\int_\ZT \frac{1-r^2}{1-2r\cos (x-t)+r^2} d\mu(t)\
\emd
converges non-tangentially to $\mu'(x_0)$ as $r\to 1$.
\end{OldTheorem}

These are two fundamental theorems, having many applications in different mathematical theories (analytic functions, Hardy spaces, harmonic analysis, differential equations and etc ).
J. Littlewood \cite{Lit} made an important complement to these theorems, proving essentiality of nontangential approach in Fatou's theorems.
\begin{OldTheorem}[Littlewood, 1927]
Let $\gamma\subset \overline{D}$ be an arbitrary curve, internally tangent at $z=1$ and no having other common point with $|z|=1$. Let $\gamma_x$ be the rotation of $\gamma$ about the origin by $e^{ix}$. Then there exists a bounded analytic function $f(z)$, $z\in D$, which does not have boundary limit along $\gamma_x$ for almost every $e^{ix}$.
\end{OldTheorem}
There are various generalization of these theorems in different aspects. Lohwater and Piranian \cite{LoPi} improved Littlewood's theorem, replacing  almost everywhere divergence to everywhere. H. Aikawa in \cite{Aik1}, \cite{Aik2} extended Littlewood's theorem for harmonic functions. Almost everywhere convergence over some semi-tangential regions investigated by Nagel and Stein \cite{NaSt}, Di Biase \cite{DiBi}, Di Biase-Stokolos-Svensson-Weiss \cite{BSSW}.
P.~Sj\"{o}gren (\cite{Sog1}, \cite{Sog2}, \cite{Sog3}), J.-O.~R\"{o}nning (\cite{Ron1}, \cite{Ron2}, \cite{Ron3}),  I.~N.~Katkovskaya and V.~G.~ Krotov (\cite{Kro}) obtained some tangential convergence properties for the square root Poisson integral. Unfortunately, we are not able to talk about the details of  these investigations within this paper. Some of them will be discussed in the last section.

We define $\lambda(r)$-convergence, which is a generalization of nontangential convergence in the unit disc. Let $\lambda(r):(0,1)\to \ZR_+$ be a function with $\lambda(r)\searrow 0$ as $r\to 1$.  For a given $x\in \ZT$ we define $\lambda(r,x)$ to be the interval $[x-\lambda(r), x+\lambda(r)]$. If $\lambda(r)\ge \pi$ we assume that $\lambda(r,x)=\ZT$. Let $F_r(x)$ be a family of functions from $L^1(\ZT)$, where $r$ varies in $(0,1)$. We say $F_r(x)$  is $\lambda(r)$-convergent at a point $x\in \ZT$ to a value $A$, if
\md0
\lim_{r\to 1}\sup_{\theta\in \lambda(r,x)}|F_r(\theta)-A|=0.
\emd
Otherwise this relation will be denoted by
\md0
\lim_{\stackrel{r\to 1}{\theta\in\lambda(r,x) }}F_r(\theta)=A.
\emd
It is clear, that the non-tangential convergence in the unit disc is the case of $\lambda(r)=c(1-r)$.

Given function of bounded variation $\mu(t)$ defines Borel measure on $\ZT$. We consider the family of integrals
\md1\label{Phi}
\Phi_r(x,d\mu)=\int_\ZT \varphi_r(x-t)d\mu(t), \quad 0<r<1,
\emd
where kernels $\varphi_r(x)\in L^\infty (\ZT)$ form an approximative of identity (AI), that is

\begin{prop}
$\int_\ZT \varphi_r(t)dt\to 1\hbox{ as }r\to 1,$
\end{prop}
\begin{prop}
$\varphi_r^*(x)=\sup_{|x|\le |t|\le \pi}|\varphi_r(t)|\to 0 \hbox{ as } r\to 1, 0<|x|<\pi,$
\end{prop}
\begin{prop}
$\sup_{0<r<1}\int_\ZT\varphi_r^*(x)<\infty.$
\end{prop}

\noindent
If $\mu(t)$ is absolute continuous and $d\mu(t)=f(t)dt$, $f\in L^1(\ZT)$, then \e{Phi} will be denoted by $\Phi_r(x,f)$. We shall prove that the condition
\md0
\limsup_{r\to 1}\lambda(r) \|\varphi_r\|_\infty<\infty
\emd
is necessary and sufficient for almost everywhere $\lambda(r)$-convergence of the integrals $\Phi_r(x,d\mu)$ as well as $\Phi_r(x,f)$, $f\in L^1(\ZT)$. Moreover we prove that convergence holds at any point where $\mu(t)$ is differentiable. An analogous
necessary and sufficient condition will be established also for almost everywhere $\lambda(r)$-convergence of $\Phi_r(x,f)$ if $f\in L^\infty(\ZT)$, and this condition looks like
\md0
\lim_{\tau\to 0}\left(\limsup_{\delta \to 0}\sup_{\tau<r<1} \int_{-\delta \lambda(r)}^{\delta \lambda(r)}\varphi_r(t)dt\right)=0.
\emd
If $\varphi_r$ coincides with the Poisson kernel, then $\|\varphi_r\|_\infty=O(1/(1-r))$ and from this results we deduce Fatou's theorems. Other consequences will be discussed in the last section.
\section{Fatou type theorems: the case of bounded mesures}

We denote by $\BV(\ZT)$ the functions of bounded variation on $\ZT$. We say that the given approximation of identity $\{\varphi_r(x)\}$ is regular if each $\varphi_r(x)$ is positive, decreasing on $[0,\pi]$ and increasing on $[-\pi, 0]$. In this case the property 3) is unnecessary, because it will immediately follows from 1).
\begin{theorem}\label{T1}
Let $\{\varphi_r\}$  be a regular AI and $\lambda(r)$ satisfies the condition
\md1\label{TV}
\limsup_{r\to 1}\lambda(r) \|\varphi_r\|_\infty<\infty.
\emd
If $\mu(t)\in \BV( \ZT)$  is differentiable at $x_0$, then
\md0
\lim_{\stackrel{r\to 1}{x\in\lambda(r,x_0) }}\Phi_r\left(x,d\mu\right)=\mu'(x_0).
\emd
\end{theorem}
An analogous theorem holds as well in the non-regular case of kernels, but at this time the points where \e{Phi} converges satisfy strong differentiability condition. We say a function of bounded variation is strong differentiable at $x_0\in \ZT$, if
the there exist a number $c$ such that the variation of the function $\mu(x)-cx$ has zero derivative at $x_0$. If $d\mu(t)=f(t)dt$ this property means that $x_0$ is Lebesgue point for $f(x)$, that is
\md0
\lim_{h\to 0}\frac{1}{2h}\int_{-h}^h|f(x)-f(x_0)|dx=0.
\emd
It is well-known that strong differentiability at $x_0$ implies the existence of $\mu'(x_0)$, and any function of bounded variation is strong differentiable almost everywhere.
\begin{theorem}\label{T2}
Let $\{\varphi_r\}$  be an arbitrary AI and $\lambda(r)$ satisfies the condition \e{TV}. If $\mu(t)\in\BV(\ZT)$ is strong differentiable at $x_0\in \ZT$, then
\md0
\lim_{\stackrel{r\to 1}{x\in\lambda(r,x_0) }}\Phi_r\left(x,d\mu \right)=\mu'(x_0).
\emd
\end{theorem}
The following theorem implies the sharpness of the condition \e{TV} in \trm{T1} and \trm{T2}.
\begin{theorem}\label{T3}
If $\{\varphi_r(x)\ge 0\}$ is an AI and the function $\lambda(r)$ satisfies the condition
\md1\label{2-1}
\limsup_{r\to 1}\lambda(r) \|\varphi_r\|_\infty=\infty,
\emd
then there exist a positive function $f\in L^1(\ZT)$  such that
\md1\label{2-2}
\limsup_{\stackrel{r\to 1}{y\in\lambda(r,x) }}\Phi_r\left(y,f\right)=\infty
\emd
for all $x\in \ZT$.
\end{theorem}

The following lemma plays significant role in the proofs of \trm{T1} and \trm{T2}.
\begin{lemma}\label{L1}
Let a positive function $\varphi(t)\in L^\infty(\ZT)$ is decreasing on $[0,\pi]$ and increasing on $[-\pi,0]$. Then for any numbers $\varepsilon\in (0,1)$ and $\theta\in (-\pi,\pi)$ there exist a finite family of intervals $I_j\subset \ZT$, $j=1,2,\ldots ,n$, containing $0$ in their closures $\bar I_j$,  and numbers $\varepsilon_j=\pm \varepsilon$ such that
\md2
&|I_j|\le 2\sup\{|t|:\, \varphi(t)\ge\varepsilon \},\quad j=1,2,\ldots,n,\\
&\sum_{j=1}^n|I_j|<10\varepsilon^{-1}\max\{1,|\theta|\cdot \|\varphi \|_\infty,\|\varphi \|_1\},\\
&\left|\varphi(x-\theta)-\sum_{j=1}^n\varepsilon_j \ZI_{I_j}(x)\right|\le\varepsilon.
\emd
\end{lemma}
\begin{proof}
Denote
\md2
&y_k=\sup\{t>0:\, \varphi(t)\ge\varepsilon k\},\\
 &x_k=\sup\{t>0:\, \varphi(-t)\ge\varepsilon k\},\quad
 k=1,2,\ldots, l=\left[\frac{\|\varphi\|_\infty}{\varepsilon}\right].
\emd
Then we obviously have
\md3
&y_0=\pi,\quad 0\le y_l\le y_{l-1}\le  \ldots \le y_1\le \sup\{|t|:\, \varphi(t)\ge\varepsilon \},\label{2-3}\\
&x_0=\pi,\quad 0\le x_l\le x_{l-1}\le  \ldots \le x_1\le \sup\{|t|:\, \varphi(t)\ge\varepsilon \},\label{2-4}\\
&\left|\varphi(x-\theta)-\varepsilon\sum_{k=1}^l\ZI_{(\theta-x_k,\theta+y_k)}(x)\right|\le \varepsilon.\label{2-5}
\emd
Without loss of generality we can suppose $0\le \theta<\pi$. Then we denote
\md0
k_0=\max\{ k:\, 0\le k \le l,\,\theta-x_k\le 0\}.
\emd
We define the desired intervals $I_j$, $j=1,2,\ldots, n=2l-k_0$, by
\begin{equation*}
I_j=\left\{
\begin{array}{lcl}
(\theta-x_j,\theta+y_j) &\hbox{ if }& j\le k_0,\\
(0,\theta+y_j) &\hbox{ if }& k_0<j\le l,\\
(0,\theta-x_{j-l+k_0}]&\hbox{ if }& l<j\le n=2l-k_0.
\end{array}
\right.
\end{equation*}
Using the equality
\md8
\ZI_{(\theta-x_k,\theta+y_k)}(x)&\\
=\ZI_{(0,\theta+y_k)}&(x)-\ZI_{(0,\theta-x_k]}(x)=\ZI_{I_k}(x)-\ZI_{I_{k+l-k_0}}(x),\quad k_0<k\le l,
\emd
we get
\md1\label{2-6}
\varepsilon \sum_{k=1}^l\ZI_{(\theta-x_k,\theta+y_k)}(x)=\sum_{j=1}^n\varepsilon_j\ZI_{I_j}(x),
\emd
where
\begin{equation}\label{2-7}
\varepsilon_j=\left\{
\begin{array}{rcl}
\varepsilon &\hbox{ if }& 1\le j\le l,\\
-\varepsilon&\hbox{ if }& l<j\le n.
\end{array}
\right.
\end{equation}
We note that $\varepsilon_j=-\varepsilon$ in the case when $I_j$ coincides with one of the intervals $(0,\theta-x_k]$, $k_0<k\le l$. Hence we have
\md1\label{2-8}
\sum_{j=l+1}^n|I_j|=\sum_{k=k_0+1}^l(\theta-x_k)\le l\cdot \theta\le \frac{\theta \|\varphi\|_\infty}{\varepsilon}.
\emd
From \e{2-5} and \e{2-6} we get
\md1\label{2-9}
\left|\varphi(x-\theta)-\sum_{j=1}^n\varepsilon_j\ZI_{I_j}(x)\right|\le \varepsilon
\emd
and therefore by \e{2-7} we obtain
\md0
\left|\int_\ZT \varphi(t)dt-\varepsilon \sum_{ j=1}^l |I_j|+\varepsilon \sum_{ j=l+1}^n |I_j|\right|\le 2\pi \varepsilon<2\pi.
\emd
This and \e{2-8} imply
\md0
\varepsilon\sum_{ j=1}^n |I_j|\le 2\varepsilon \sum_{ j=l+1}^n |I_j|+\|\varphi\|_1+2\pi  \le 2\theta \|\varphi\|_\infty +\|\varphi\|_1+2\pi,
\emd
which together with \e{2-3}, \e{2-4} and \e{2-9} completes the proof of lemma.
\end{proof}
\begin{proof}[Proof of \trm{T1}]
Without loss of generality we may assume that $x_0=0$ and $\mu'(x_0)=0$. We fix a function $\theta(r):(0,1)\to\ZR$ with $|\theta(r)|\le \lambda(r)$.  From \e{TV} we get
\md1\label{2-10}
|\theta(r)|\cdot \|\varphi_r\|_\infty\le M,\quad r_0<r<1.
\emd
 Using the property 2) of the kernels $\{\varphi_r(t)\}$  we may define a collection of numbers $\varepsilon_r>0$ such
that
\md1\label{2-11}
\varepsilon_r\searrow 0,\quad \delta_r=\sup\{|t|:\, \varphi_r(t)\ge\varepsilon_r \}\to 0\hbox { as } r\to 1.
\emd
Applying \lem{L1}, for any $0<r<1$ we define a family of intervals $\{I_j^{(r)},\, j=1,2,\ldots ,n_r\}$ such that
\md5
|I_j^{(r)}|\le2\delta_r,\quad j=1,2,\ldots,n_r,\label{2-12}\\
\sum_{j=1}^{n_r}\left|I_j^{(r)}\right|<10(\varepsilon_r)^{-1}\max\{1,|\theta(r)|\cdot\|\varphi_r\|_\infty,\|\varphi_r \|_1\},\label{2-13}\\
\left|\varphi_r(\theta(r)-t)-\sum_{j=1}^{n_r}\varepsilon_j^{(r)} \ZI_{I_j^{(r)}}(t)\right|\le\varepsilon_r,\label{2-14}
\emd
where  $\varepsilon_j^{(r)}=\pm \varepsilon_r$. From \e{2-10} and \e{2-13} we conclude
\md1\label{2-15}
\varepsilon_r\cdot \sum_{j=1}^{n_r}\left|I_j^{(r)}\right|\le L,\quad r_0<r<1,
\emd
where $L$ is a positive constant. From \e{2-11} and \e{2-14} we obtain
\md1\label{2-16}
\Phi_r(\theta(r),d\mu)=\int_\ZT\varphi_r(\theta(r)-t)d\mu(t)=\sum_{j=1}^{n_r}\varepsilon_j^{(r)} \int_{I_j^{(r)}}d\mu(t)+o(1),
\emd
where $o(1)\to 0$  as $r\to 1$.  Using this, we get
\md0
\left|\Phi_r(\theta(r),d\mu)\right|\le \varepsilon_r\cdot \sum_{j=1}^{n_r}\left|I_j^{(r)}\right|\cdot \frac{1}{\left|I_j^{(r)}\right|}\left|\int_{I_j^{(r)}}d\mu(t)\right|+o(1).
\emd
According to \e{2-11} and \e{2-12}, we have
\md0
\max_{1\le j\le n_r}\frac{1}{\left|I_j^{(r)}\right|}\left|\int_{I_j^{(r)}}d\mu(t)\right|\to \mu'(0)=0 \hbox{ as } r\to 1.
\emd
This together with \e{2-15} and \e{2-16} implies that $\Phi_r(\theta(r),d\mu)\to 0$ as $r\to 1$.
\end{proof}
\begin{proof}[Proof of \trm{T2}]
 Let $\theta(r)$ satisfies \e{2-10}. We again assume that $x_0=0$, $\mu'(x_0)=0$ and so we will have $|\mu|'(0)=0$. Then, repeating  the same process of the proof of \trm{T1} at this time for the functions $\varphi_r^*(t)$ together with the measure $|\mu|$, instead of \e{2-16} we obtain
\md0
\int_\ZT \varphi_r^*\left(\theta(r)-t\right)d|\mu|(t)=\sum_{j=1}^{n_r}\varepsilon_j^{(r)} \int_{I_j^{(r)}}d|\mu|(t)+o(1).
\emd
Then we get
\md8
\left|\Phi_r(\theta(r),d\mu)\right|&\le \int_\ZT \varphi_r^*(\theta(r)-t)d|\mu|(t)\\
&= \varepsilon_r\cdot \sum_{j=1}^{n_r}|I_j^{(r)}|\cdot \frac{1}{|I_j^{(r)}|}\int_{I_j^{(r)}}d|\mu|(t)+o(1).
\emd
Since $|\mu|(t)$ is differentiable at $0$, we get
\md0
\Phi_r(\theta(r),d\mu)\to 0.
\emd
\end{proof}
\begin{proof}[Proof of \trm{T3}]
For any $0<r<1$ there exist a point $x_r\in (0,\pi)$, a number $0<\delta_r<\lambda(r)/4$ and a measurable set $E_r$ such that
\md3
E_r\subset (x_r-\delta_r,x_r+\delta_r),\quad |E_r|>\frac{3\delta_r}{2},\label{2-17}\\
\varphi_r(x)>\frac{\|\varphi_r\|_\infty}{2},\quad x\in E_r.\label{2-18}
\emd
From these relations it follows that $\varphi_r^*(x)>\|\varphi_r\|_\infty/2$ if $x\in(0,|x_r|)$. On the other hand, by property 3) we have $\|\varphi_r^*\|_1\le M$ for some constant $M>0$. Thus we get
\md1\label{2-19}
|x_r|\le \frac{2M}{\|\varphi_r\|_\infty},\quad 0<r<1.
\emd
Denote
\md5
n(r)=\left[\frac{4\pi}{\lambda(r)}\right]\in \ZN,\label{2-20}\\
\Delta_r=\bigcup_{k=0}^{n(r)-1}\left[\frac{2\pi k}{n(r)}-\delta_r,\frac{2\pi k}{n(r)}+\delta_r\right],\label{2-21}
\emd
and consider the function
\md0
f_r(x)=\frac{\ZI_{\Delta_r(x)}}{|\Delta_r|}=\frac{\ZI_{\Delta_r}(x)}{{2\delta_r} n(r)}.
\emd
If $x\in \ZT$ is an arbitrary point, then
\md0
x\in \left[\frac{2\pi k_0}{n(r)},\frac{2\pi (k_0+1)}{n(r)}\right)
\emd
for some $k_0$, $0\le k_0<n(r)$. Taking $\theta=x-x_r-2\pi k_0/n(r)$, from \e{2-19} and \e{2-20} we obtain
\md1\label{2-22}
|\theta |<\frac{2\pi}{n(r)}+|x_r| <\frac{2\pi\lambda(r)}{4\pi-\lambda(r)}+\frac{2M}{\|\varphi_r\|_\infty}=\lambda(r)\left(\frac{1}{2}+\frac{2M}{\lambda(r)\|\varphi_r\|_\infty}\right).
\emd
Using \e{2-1}, we may fix a sequence $r_k\nearrow 1$ such that
\md1\label{2-23}
\lambda(r_k)\|\varphi_{r_k}\|_\infty>M\cdot 4^k,\quad  k=1,2,\ldots .
\emd
From \e{2-22} and \e{2-23} we conclude
\md1\label{2-24}
|\theta |<\lambda(r),\hbox{ if } r=r_k.
\emd
Since $\varphi_r(x)\ge 0$, using \e{2-18} and \e{2-20}, for the same $x$ we get
\md9\label{2-25}
\Phi_r(x-\theta,f_r)&\\
&=\int_{\ZT} \varphi_r(2\pi k_0/n(r)+x_r-t)f_r(t)dt\\
&\ge\frac{1}{2{\delta_r} n(r)}\int_{2\pi k_0/n(r)-{\delta_r}}^{2\pi k_0/n(r)+{\delta_r}} \varphi_r(2\pi k_0/n(r)+x_r-t)dt\\
&=\frac{1}{2{\delta_r} n(r)}\int_{x_r-\delta_r}^{x_r+\delta_r} \varphi_r(u)du\\
&\ge \frac{1}{2{\delta_r} n(r)}\cdot \frac{3\delta_r}{2}\cdot\frac{\|\varphi_r\|_\infty}{2}\ge \frac{3\lambda(r)\|\varphi_r\|_\infty}{16}.
\emd
From \e{2-23}, \e{2-24} and \e{2-25} we obtain
\md1\label{2-26}
\sup_{\theta\in\lambda(r_k,x)}\Phi_{r_k}(\theta,f_{r_k})\ge 3M\cdot 4^{k-2},\quad x\in\ZT.
\emd
Using \e{2-23}, we define
\md0
f(x)=\sum_{k=1}^\infty 2^{-k}f_{r_k}(x)\in L^1(\ZT).
\emd
From \e{2-26} and \e{2-23} it follows that
\md0
\sup_{\theta\in\lambda(r_k,x)}\Phi_{r_k}(\theta,f)\ge\sup_{\theta\in\lambda(r_k,x)}2^{-k}\Phi_{r_k}(\theta,f_{r_k})\ge 3M\cdot 2^{k-4},
\emd
which implies \e{2-2}.
\end{proof}
\section{Fatou type theorems: the case $L^\infty$}
Let $\lambda(r):(0,1)\to \ZR$ be an arbitrary real function with $\lambda(r)\to 0$ as $r\to 1$. The quantity
\md0
\gamma_\lambda=\lim_{\tau\to 1}\left(\limsup_{\delta \to 0}\sup_{\tau<r<1} \int_{-\delta \lambda(r)}^{\delta \lambda(r)}\varphi_r(t)dt\right)
\emd
completely characterizes the almost everywhere $\lambda(r)$ convergence  property of $\Phi_r(x,f)$ in $L^\infty(\ZT)$.
\begin{theorem}\label{T5}
If $\{\varphi_r(x)\}$ is a regular AI consisting of even functions and   $\gamma_\lambda=0$, then for any
$f\in L^\infty( \ZT)$ the relation
\md0
\lim_{\stackrel{r\to 1}{\theta\in\lambda(r,x) }}\Phi_r\left(\theta,f\right)=f(x)
\emd
holds at any Lebesgue point.
\end{theorem}
\begin{theorem}\label{T6}
If $\{\varphi_r(x)\}$ is a regular AI consisting of even functions and   $\gamma_\lambda>0$,  then there exists a set $E\subset \ZT$, such that
$\Phi_r\left(x,\ZI_E\right)$ is $\lambda(r)$-divergent at any $x\in\ZT$.
\end{theorem}
Note that if $\lambda(t)$ satisfies the condition \e{TV}, then $\gamma_\lambda=0$. One can easily construct a family of kernels $\{\varphi_r\}$ such that $\gamma_\lambda=0$, but \e{TV} is not satisfied. This means the $\gamma_\lambda=0$ is a weaker condition than \e{TV}.
\begin{proof}[Proof of \trm{T5}]
Since $\gamma_\lambda=0$ for any $0<\varepsilon <1/2$ we may chose $\delta>0$ and $0<\tau<1$, such that
\md1\label{3-1}
\int_{-\delta \lambda(r)}^{\delta \lambda(r)}\varphi_r(t)dt<\varepsilon,\quad \tau<r<1.
\emd
Then we define
\md2
\varphi^{(1)}_r(x)=\left\{
\begin{array}{lrl}
\varphi_r(x)-\varphi_r(\delta \lambda(r))& \hbox { if }&  |x|\le \delta \lambda(r),\\
0&\hbox { if }& \delta \lambda(r)<|x|<\pi.
\end{array}
\right.
\emd
and
\md0
\varphi^{(2)}_r(x)=\frac{\varphi_r(x)-\varphi^{(1)}_r(x)}{\|\varphi_r(x)-\varphi^{(1)}_r\|_{L^1}}
=\frac{\varphi_r(x)-\varphi^{(1)}_r(x)}{1-l_r}
\emd
where
\md0
l_r=\int_{-\delta \lambda(r)}^{\delta \lambda(r)}\left(\varphi_r(t)-\varphi_r(\delta \lambda(r)\right)dt<\varepsilon<\frac{1}{2},\quad \tau<r<1.
\emd
It is clear, that $\{\varphi^{(2)}_r\}$ is a regular AI and we have
\md1\label{3-2}
\varphi_r(x)=\varphi_r^{(1)}(x)+(1-l_r)\varphi_r^{(2)}(x).
\emd
From \e{3-1} it follows that
\md1\label{3-3}
\left|\int_\ZT\varphi^{(1)}_r(x-t)f(t)dt\right|\le \|f\|_\infty\int_{-\delta \lambda(r)}^{\delta \lambda(r)}\varphi_r(t)dt\le \varepsilon\|f\|_\infty
\emd
and
\md0
\varphi_r(\delta \lambda(r))\cdot 2\delta \lambda(r)<\varepsilon,\quad \tau<r<1.
\emd
Thus,using the definition of $\varphi^{(2)}_r(x)$, we get
\md0
\|\varphi_r^{(2)}\|_\infty\cdot  \lambda(r)<\frac{\varepsilon}{2\delta(1-l_r)}<\frac{\varepsilon}{4\delta}
\emd
Using this and \trm{T1} we conclude, that
\md1\label{3-4}
\lim_{\stackrel{r\to 1}{\theta\in\lambda(r,x) }}\int_\ZT\varphi^{(2)}_r(\theta-t)f(t)dt=f(x)
\emd
at any Lebesgue point.
Now without loss of generality we assume that $f(x)\ge 0$. If $x$ is an arbitrary Lebesgue point, using \e{3-2}, \e{3-3} and \e{3-4} we get
 \md2
 &\limsup_{\stackrel{r\to 1}{\theta\in\lambda(r,x) }}\Phi_r(\theta,f)\le \varepsilon \|f\|_\infty+f(x),\\
 &\liminf_{\stackrel{r\to 1}{\theta\in\lambda(r,x) }}\Phi_r(\theta,f)\ge -\varepsilon \|f\|_\infty+(1-\varepsilon)f(x).
 \emd
 Since $\varepsilon $ can be taken sufficiently small, we get
 \md0
 \lim_{\stackrel{r\to 1}{\theta\in\lambda(r,x) }}\Phi_r(\theta,f)=f(x),
 \emd
 and the theorem is proved.
\end{proof}
If $f(x)$ is a function defined on a set $E\subset \ZT$ we denote
\md0
\OSC_{x\in E}f(x)=\sup_{x,y\in E}|f(x)-f(y)|.
\emd
\begin{lemma}\label{L1}
Let
\md0
U_n^\delta=\bigcup_{k=0}^{n-1}\left(\frac{\pi (2k+1-\delta)}{n},\frac{\pi (2k+1+\delta)}{n}\right),\quad n\in \ZN,\quad 0<\delta<\frac{1}{2},
\emd
and $J\subset \ZT$, $\pi >|J|\ge 16\pi/n$, is an arbitrary closed interval. If a measurable set $E\subset \ZT$ satisfies either
\md0
E\cap J=J\cap U_n^\delta \hbox { or } E\cap J=J\setminus U_n^\delta,
\emd
and  $\varphi (x)\in L^\infty(\ZT)$ is an even decreasing on $[0,\pi]$ function, then
\md1\label{3-5}
\OSC_{\theta\in \left[x-\frac{4\pi}{n},x+\frac{4\pi}{n}\right]}\int_\ZT\varphi(\theta-t)\ZI_E(t)dt>\int_{-\frac{\pi\delta}{n}}^{\frac{\pi\delta}{n}}\varphi(t)dt-16\delta-2\pi\varphi\left(\frac{|J|}{4}\right),
\emd
for any $x\in J$.
\end{lemma}

\begin{proof}
We suppose $J=[a,b]$ and
\md0
\frac{2\pi (p-1)}{n}<a\le \frac{2\pi p}{n},\quad \frac{2\pi (q-1)}{n}<b\le \frac{2\pi q}{n}.
\emd
First we consider  the case
\md1\label{3-6}
E\cap J=J \cap U_n^\delta .
\emd
If $x\in J$, then
\md0
x\in I=\left[\frac{2\pi(m-1)}{n},\frac{2\pi m}{n} \right]
\emd
for some $p\le m\le q$. Without loss of generality we may assume that the center of $I$ is on the left hand side of the center of $J$. Then we will have
\md1\label{3-7}
b-\frac{ 2\pi(m+1)}{n}\ge \frac{|J|}{2}-\frac{ 4\pi}{n}\ge \frac{|J|}{4}.
\emd
It is clear, that the points
\md1\label{3-8}
\theta_1=\frac{ 2\pi m}{n}+\frac{\pi}{n},\quad \theta_2=\frac{ 2\pi(m+1)}{n}
\emd
are in the interval $[0,x+4\pi/n]$. Besides we have
\md8
\int_\ZT\varphi(\theta_1-t)\ZI_E(t)dt&-\int_\ZT\varphi(\theta_2-t)\ZI_E(t)dt\\
&=\int_{\theta_2-\pi}^a[\varphi(\theta_1-t)-\varphi(\theta_2-t)]\ZI_E(t)dt\\
&+\int_a^b[\varphi(\theta_1-t)-\varphi(\theta_2-t)]\ZI_E(t)dt\\
&+\int_b^{\theta_2+\pi}[\varphi(\theta_1-t)-\varphi(\theta_2-t)]\ZI_E(t)dt\\
&=A_1+A_2+A_3.
\emd
Since $\varphi$ is decreasing on $[0,\pi]$ we have
\md1\label{3-9}
A_1\ge 0.
\emd
If $t\in [b,\theta_2+\pi]$ then, using \e{3-7}, we get
\md0
t-\theta_2\ge b-\theta_2\ge \frac{|J|}{4},\quad t-\theta_1\ge  \frac{|J|}{4} ,
\emd
which implies
\md1\label{3-10}
|A_3|\le 2\pi \varphi\left(\frac{|J|}{4}\right).
\emd
To estimate $A_2$ we denote
\md1\label{3-11}
a_k=\int_{\pi( k-\delta)/n}^{\pi( k+\delta)/n}\varphi(t)dt,\quad k\in \ZZ.
\emd
We have
\md1\label{3-12}
a_0=\int_{-\pi\delta/n}^{\pi\delta/n}\varphi(t)dt.
\emd
Using properties of $\varphi $ we  have $a_k=a_{-k}$ and $a_1\ge a_2\ge\ldots $. Using Chebishev inequality we have $\varphi(t)\le 1/t$. Thus we obtain
\md1\label{3-13}
a_k\le a_1=\int_{\pi( 1-\delta)/n}^{\pi( 1+\delta)/n}\varphi(t)dt\le \frac{2\pi\delta/n}{\pi(1-\delta)/n}=\frac{2\delta}{1-\delta}<4\delta,\quad k\ge 1.
\emd
Using \e{3-6}, \e{3-11}, \e{3-12} and \e{3-13}, we get
\md8
A_2&\ge \sum_{k=p}^{q-2}\int_{\pi (2k+1-\delta)/n}^{\pi (2k+1+\delta)/n}[\varphi(\theta_1-t)-\varphi(\theta_2-t)]dt-8\delta\\
&=\sum_{k=p}^{q-1}\int_{\pi(2(m-k)-\delta)/n}^{\pi(2(m-k)+\delta)/n}\varphi(t)dt
-\sum_{k=p}^{q-1}\int_{\pi(2(m-k)+1-\delta)/n}^{\pi(2(m-k)+1+\delta)/n}\varphi(t)dt-8\delta\\
&=\sum_{k=m-q+1}^{m-p}a_{2k}- \sum_{k=m-q+1}^{m-p}a_{2k+1}-8\delta\ge a_0-a_1-a_{-1}-8\delta\\
&>\int_{-\pi\delta/n}^{\pi\delta/n}\varphi(t)dt-16\delta.
\emd
Combining this with \e{3-9} and \e{3-10}, we get
\md0
\int_\ZT\varphi(\theta_1-t)\ZI_E(t)dt-\int_\ZT\varphi(\theta_2-t)\ZI_E(t)dt\ge \int_{-\frac{\pi\delta}{n}}^{\frac{\pi\delta}{n}}\varphi(t)dt-16\delta-2\pi\varphi\left(\frac{|J|}{4}\right),
\emd
which together with \e{3-8} implies \e{3-5}.  To deduce the case $E\cap J=J\setminus U_n^\delta$ notice, that for the complement $E^c$ we have
$E^c\cap J=J\cap U_n^\delta$ and so \e{3-5} holds for $E^c$. Therefore we obtain
\md8
\OSC_{\theta\in \left[x-\frac{4\pi}{n},x+\frac{4\pi}{n}\right]}&\int_\ZT\varphi(\theta-t)\ZI_E(t)dt\\
&=\OSC_{\theta\in \left[x-\frac{4\pi}{n},x+\frac{4\pi}{n}\right]}\left(\|\varphi \|_1-\int_\ZT\varphi(\theta-t)\ZI_E(t)dt\right)\\
&=\OSC_{\theta\in \left[x-\frac{4\pi}{n},x+\frac{4\pi}{n}\right]}\left(\int_\ZT\varphi(\theta-t)\ZI_{E^c}(t)dt\right)\\
&>\int_{-\frac{\pi\delta}{n}}^{\frac{\pi\delta}{n}}\varphi(t)dt-16\delta-2\pi\varphi\left(\frac{|J|}{4}\right),
\emd
which completes the proof of the lemma.
\end{proof}
\begin{proof}[Proof of \trm{T6}]
Since $\gamma_\lambda>0$, there exist sequences $\delta_k\searrow 0$ and $r_k\to 1$, such that
\md1\label{3-0}
\int_{-\delta_k \lambda(r_k)}^{\delta_k \lambda(r_k)}\varphi_{r_k}(t)dt>\frac{\gamma_\lambda}{2}, \quad k=1,2,\ldots .
\emd
Denote
\md1\label{3-14}
U_k=U_{n_k}^{\delta_k},\quad n_k=\left[\frac{\pi}{\lambda(r_k)}\right],
\emd
where $U_n^\delta$ is defined in the \lem{L1}.
Define the sequences of measurable sets $E_n$ by
\md0
E_1=U_1,\quad E_k=E_{k-1}\bigtriangleup U_k=(E_{k-1}\setminus U_k)\cup(U_k\setminus E_{k-1}),\quad k>1
\emd
We say $J$ is an adjacent interval for $E_k$, if it is a maximal interval containing either in $E_k$ or $(E_k)^c$. The family of all this intervals form a covering of whole $\ZT$.
It is easy to observe, that a suitable selection of  $\delta_k$ and $r_k$ may provide
\md3
&\varphi_{r_k}\left(\frac{|J|}{4}\right)<\frac{\gamma_\lambda}{16\pi},\hbox { if } J \hbox{ is adjacent for } E_{k-1},\label{3-15}\\
&\delta_j\le \frac{\gamma_\lambda}{2^{j+5}\|\varphi_{r_k}\|_\infty},\quad j\ge k+1,\label{3-16}
\emd
It is easy to observe, that if $k<m$, then
\md1\label{3-17}
\|\ZI_{E_k}-\ZI_{E_m}\|_1=|E_k\bigtriangleup E_m|\le \sum_{j\ge k+1}|U_j|
\emd
This implies, that $\ZI_{E_n}(t)$ converges to a function $f(t)$ in $L^1$. Using Egorov's theorem, we conclude that $f(t)=\ZI_E(t)$ for a measurable set $E\subset \ZT$. Tending  $m$ to infinity,  from \e{3-16} and \e{3-17} we get
\md1\label{3-18}
|E_k\bigtriangleup E|\le\left|\bigcup_{j\ge k+1} U_j\right|\le 2\pi\sum_{j\ge k+1}\delta_j\le \frac{\gamma_\lambda}{16\|\varphi_{r_k}\|_\infty}.
\emd
Fix a point $x\in\ZT$. We have $x\in J$ where $J$ is an adjacent interval for $E_{k-1}$.
From the definition of $E_k$ it follows that either
\md0
E_k\cap J=J\cap U_k \hbox { or } E_k\cap J=J\setminus U_k.
\emd
From \e{3-14} we have
\md0
\lambda(r_k,x)=(x-\lambda(r_k),x+\lambda(r_k))\subset \left[x-\frac{4\pi}{n_k},x+\frac{4\pi}{n_k}\right].
\emd
Thus, applying \lem{L1}, \e{3-0} and \e{3-15}, we get
\md8
\OSC_{\theta\in \lambda(r_k,x)}&\Phi_{r_k}(\theta,\ZI_{E_k})\\
&\ge\OSC_{\theta\in \left[x-\frac{4\pi}{n_k},x+\frac{4\pi}{n_k}\right]}\Phi_{r_k}(\theta,\ZI_{E_k})\\
&\ge\int_{-\frac{\pi\delta_k}{n_k} }^{\frac{\pi\delta_k}{n_k}}\varphi_{r_k}(t)dt-16\delta_k-2\pi\varphi_{r_k}\left(\frac{|J|}{4}\right)\\
&\ge\int_{-\delta_k \lambda(r_k)}^{\delta_k \lambda(r_k)}\varphi_{r_k}(t)dt-16\delta_k-\frac{\gamma_\lambda}{8}\\
&\ge\frac{\gamma_\lambda}{4}-16\delta_k.
\emd
From \e{3-18} we conclude
\md8
\OSC_{\theta\in \lambda(r_k,x)}&\Phi_{r_k}(\theta,\ZI_E)\\
&>\OSC_{\theta\in \lambda(r_k,x)}\Phi_{r_k}(\theta,\ZI_{E_k})-\frac{\gamma_\lambda}{16}\ge \frac{\gamma_\lambda}{8}-16\delta_k,
\emd
which completes the proof of theorem since $\delta_k\to 0$.
\end{proof}
\section{Final remarks}

In the definition of $\lambda(r)$-convergence the range of the parameter $r$ is $(0,1)$ with the limit point $1$. Certainly it is not essential in the theorems. We could take any set $Q\in\ZR$ with limit point $r_0$ which is either a finite number or $\infty$. We may define an approximation of the identity on the real line  to be a family of functions $\varphi_r\in L^\infty(\ZR)\cap L^1(\ZR)$, $r>0$,  satisfying the same conditions 1)-3) as AI on $\ZT$ has. We have just make a little change in the condition 2), that is $\left\|\varphi_r^*\cdot \ZI_{\{|t|\ge \delta\}}\right\|_1\to 0$ as $r\to 0$, $\delta>0$. In this case usually convergence is considered while $r\to 0$. Analogous results can be formulated and proved for the integrals
\md1\label{4-1}
\Phi_r(x,d\mu)=\int_\ZR \varphi_r(t-x)d\mu(t), \quad r>0.
\emd
And it can be done just repeating the above proofs after miserable changes.

Any function $\Phi(x)\in L^\infty(\ZR)\cap L^1(\ZR)$ defines an approximation of identity by $\varphi_r(x)=r\Phi(x/r)$ as $r\to 0$. The operators corresponding to such kernels in higher dimensional case investigated by E.~M.~Stain  (\cite{Ste}, p. 57). We note for such kernels we have $\|\varphi_r\|_\infty=r^{-1}\|\Phi\|_\infty$ and therefore the condition \e{TV} takes the form $\lambda(r)\le cr$. This bound  characterizes the nontangential convergence in the upper half plane and  it turns out to be a necessary and sufficient condition for almost everywhere $\lambda(r)$-convergence of the integrals \e{4-1}.

P.~Sj\"{o}gren (\cite{Sog1}, \cite{Sog2}, \cite{Sog3}),  J.-O.~R\"{o}nning (\cite{Ron1}, \cite{Ron2}, \cite{Ron3}), I.~N.~Katkovskaya and V.~G.~ Krotov (\cite{Kro}) considered the square root Poisson integrals
\md1\label{a1}
\zP_r^0(x,f)=\frac{1}{ c(r)}\int_\ZT [P_r(x-t)]^{1/2} f(t)dt,
\emd
where
\md0
c(r)=\int_\ZT [P_r(t)]^{1/2}dt
\emd
is the normalizing coefficient. They proved, that
\md0
\lim_{r\to 1}\zP_r^0 (x+\theta(r),f)=f(x) \hbox { a.e. }
\emd
whenever $f\in L^p(\ZT)$, $1\le p\le \infty$,  and
\md2
|\theta(r)|\le \left\{
\begin{array}{lcl}
c(1-r)\left(\log\frac{1}{1-r}\right)^p&\hbox { if } &1\le p<\infty,\\
(1-r)^{1-\varepsilon},&\hbox { if } & p=\infty.
\end{array}
\right.
\emd
The case of $p=1$ is proved in \cite{Sog1}, $1<p\le \infty$ is considered in \cite{Ron1}, \cite{Ron2}. They provide also some weak type inequalities for the maximal operators of square root Poisson integrals.  In the paper (\cite{Kro}) authors obtained weighted strong type inequalities for the same operators.  The cases $p=1$ and $p=\infty$ are consequences of the \trm{T1} with an additional information about the points where the convergence occurs.

At the end of the paper we would like to bring couple of consequences of our theorems, that we consider interesting.
\begin{corollary}
If $\sigma_n(x,f)$ are the Fejer means of Fourier series of a function $f\in L^1(\ZT)$ and $\theta_n=O(1/n)$, then $\sigma_n(x+\theta_n,f)\to f(x)$ at any Lebesgue point $x\in \ZT$.
\end{corollary}
\begin{corollary}
If $S_n(x,f)$ are the partial sums of a function $f\in L^1(0,1)$ in Franklin system and  $\theta_n=O(1/n)$, then $S_n(x+\theta_n,f)\to f(x)$ at any Lebesgue point $x\in \ZT$.
\end{corollary}

 \end{document}